\documentclass[letterpaper, 10 pt, conference]{ieeeconf}  % Comment this line out
\IEEEoverridecommandlockouts                              % This command is only
\usepackage{mathtools}
\usepackage{amsfonts}
\usepackage{booktabs}
\usepackage{cite}
\usepackage[usenames,dvipsnames]{xcolor}
\usepackage{mathtools}
\usepackage{flushend}
\newcommand{\x}{\boldsymbol{x}}
\newcommand{\y}{\boldsymbol{y}}
\newcommand{\boldb}{\boldsymbol{b}}

\newcommand{\q}{\boldsymbol{q}}
\newcommand{\boldv}{\boldsymbol{v}}

\newcommand{\w}{\boldsymbol{w}}
\newcommand{\Nat}{\mathbb{N}}
\newcommand{\R}{\mathbb{R}}
\newcommand{\Resets}{\mathcal{R}}
\newcommand{\E}{\mathbb{E}}
\newcommand{\XMom}{\mathcal{X}}
\newcommand{\Leg}{\mathcal{L}}
\newcommand{\XbarMom}{\overline{\mathcal{X}}}

\usepackage{dsfont}

\usepackage{comment}

\newif\ifshow
\showtrue
\ifshow

\else
\excludecomment{scratch}
\fi

\newtheorem{remark}{\bfseries Remark}
\newtheorem{example}{\bfseries Example}

\newtheorem{theorem}{\bfseries Theorem}

\newtheorem{lemma}{\bfseries Lemma}

\usepackage{xargs}                      % Use more than one optional parameter in a new commands
\usepackage[colorinlistoftodos,prependcaption,textsize=tiny]{todonotes}
\newcommandx{\unsure}[2][1=]{\todo[linecolor=red,backgroundcolor=red!25,bordercolor=red,#1]{#2}}
\newcommandx{\change}[2][1=]{\todo[linecolor=blue,backgroundcolor=blue!25,bordercolor=blue,#1]{#2}}
\newcommandx{\info}[2][1=]{\todo[linecolor=OliveGreen,backgroundcolor=OliveGreen!25,bordercolor=OliveGreen,#1]{#2}}
\newcommandx{\improvement}[2][1=]{\todo[linecolor=Plum,backgroundcolor=Plum!25,bordercolor=Plum,#1]{#2}}
\newcommandx{\thiswillnotshow}[2][1=]{\todo[disable,#1]{#2}}

\title{\LARGE \bf
Moment Analysis of Stochastic Hybrid Systems \\ Using Semidefinite Programming}
\author{Khem Raj Ghusinga$^{1}$, Andrew Lamperski$^{2}$, and Abhyudai Singh$^{1}$% <-this % stops a space
\thanks{$^{1}$ Khem Raj Ghusinga and Abhyudai Singh are with the Department of Electrical and Computer Engineering, University of Delaware, Newark, DE, USA 19716.
        {\tt\small \{khem,absingh\}@udel.edu}}%
\thanks{$^{2}$ Andrew Lamperski  is with the Department of Electrical and Computer Engineering, University of Minnesota, Minneapolis, MN, USA 55455.
        {\tt\small alampers@umn.edu}}
        }

\begin{document}

\maketitle
\thispagestyle{empty}
\pagestyle{empty}

%%%%%%%%%%%%%%%%%%%%%%%%%%%%%%%%%%%%%%%%%%%%%%%%%%%%%%%%%%%%%%%%%%%%%%%%%%%%%%%%
\begin{abstract}
This paper proposes a semidefinite programming based method for
estimating moments of a stochastic hybrid system (SHS). For polynomial
SHSs -- which consist of polynomial continuous vector fields, reset
maps, and transition intensities -- the dynamics of moments evolve
according to a system of linear ordinary differential equations.
However, it is generally not possible to solve the system exactly
since time evolution of a specific moment may depend upon moments of
order higher than it. One way to overcome this problem is to employ
so-called moment closure methods that give point approximations to
moments, but these are limited in that accuracy of the estimations is
unknown.
We find lower and upper bounds on a moment of interest via a
semidefinite program that includes linear constraints obtained from
moment dynamics, along with semidefinite constraints that arise from
the non-negativity of moment matrices. These bounds are further shown
to improve as the size of semidefinite program is increased.
The key insight in the method is a reduction from stochastic hybrid
systems with multiple discrete modes to a single-mode hybrid system
with algebraic constraints. 
We
further extend the scope of the proposed method to a class of
non-polynomial SHSs which can be recast to polynomial SHSs via
augmentation of additional states. Finally, we illustrate the
applicability of results via examples of SHSs drawn from different
disciplines. 
\end{abstract}

%%%%%%%%%%%%%%%%%%%%%%%%%%%%%%%%%%%%%%%%%%%%%%%%%%%%%%%%%%%%%%%%%%%%%%%%%%%%
\section{Introduction}
Stochastic Hybrid System (SHS) is a mathematical framework that is applicable to a wide-array of phenomena in engineering, biological and physical systems  \cite{bohacek2003hybrid,Hespanha04,hespanha2005model, hu2006application, antunes2013stochastic, hespanha_modeling_2014, buckwar2011exact, visintini2006monte, liu2011probabilistic, hu2003aircraft, stvrelec2012modeling, david2015statistical,Li2017}. An SHS is specified by a finite number of discrete states (modes), stochastic dynamics of a continuous state, a set of rules governing transitions that can change the continuous state as well as the discrete state, and reset maps that define how the states change after a transition \cite{hespanha2006modelling,teel2014stability,Hu2000}. Despite wide-applicability of SHSs, their formal analysis is often challenging. For example, the probability density function of the SHS state space can be characterized by Kolmogorov equations, but solving them analytically is not possible in most cases. The probability density function can also be estimated by running a large number of Monte Carlo simulations; however, it is typically computationally prohibitive. 

Computing moments of SHS is another approach that provides important insights into its dynamics. For an SHS whose continuous state, transition intensities, and reset maps are described via polynomials, the time evolution of its moments is governed by a system of linear ordinary differential equations \cite{hespanha2005model}. However, the moment dynamics is not closed (except for few special cases, e.g.,  \cite{soltani2017stochastic,sos17Automatica}) as in the time-evolution of a moment of certain order depends on moments of order higher than it. Furthermore, when the SHS consists of non-polynomial nonlinearities, the moment dynamics also contains non-polynomial moments, in addition to the higher order moments as in the polynomial case. In presence of these issues, it is desirable to develop methods that provide approximate values of desired moments with provable guarantees.

For polynomial SHSs, the problem of unclosed moments is usually overcome by using the \textit{moment closure} methods \cite{whittle1957use,krishnarajah2005novel,konkoli2012modeling,smadbeck2013closure,grima2012study,Kuehn2016}. These methods truncate the infinite-dimensional moment equations to some finite order and then approximate the higher order moments appearing in them in terms of the moments of lower order. There are numerous methods proposed for this purpose which either assume that the probability density function of the state follows a certain distribution, or that some higher order moments/cumulants are zero \cite{Kuehn2016,socha2007linearization,singh2011approximate,soltani2015conditional}. A limitation of these methods is that they provide point approximations to moments of interest without any guarantee on errors. Although not-widely used in practiced, the moment closure methods are applicable to non-polynomial SHSs that can be casted to polynomial SHSs by defining additional states \cite{socha2007linearization,bcp16}.

Recently, a semidefinite programming based method to estimate moments of polynomial jump diffusion processes (and its special cases) has been developed \cite{lamperski2017analysis,lamperski2016stochastic,ghusinga2017exact}. This method utilizes the semidefinite inequalities that are satisfied by the moments of the system under consideration and finds monotonic sequence of lower and upper bounds on a moment of interest. In this paper, we extend the method to both polynomial and non-polynomial SHSs, thus covering a large class of stochastic systems. The key difference between the jump-diffusion and the SHS based models is that the SHS model has a (typically finite) number of discrete modes. While previous works have dealt with moment dynamics for multiple discrete modes, their approach has been to analyze the moments of continuous state given a discrete state. Here, we present an augmented state space method that transforms the system to a single-mode SHS and allows joint analysis of discrete and continuous states. We use a similar idea of appending additional states to write moment dynamics and estimate moments of a class of SHSs defined over non-polynomial functions. The method is illustrated via two examples drawn from communication systems, and biology.

%\subsection{Organization}
%The paper is organized as follows. Section \ref{sec:shs} provides a brief background on SHS and also describes the running examples used in the paper. Section \ref{sec:mom} deals with computation of moment dynamics for SHS, with application to the running examples. Furthermore, this section also deals with how semidefinite constraints involving moments lead to bounds on the moment dynamics. Section~\ref{sec:numerical} illustrates the joint usage of a moment closure method with the semidefinite programming based bound on the moments. Finally, Section~\ref{sec:conc} concludes the paper.

\subsection*{Notation} 
For stochastic processes and their moments, we omit explicit dependence on time unless it is not clear from the context. Inequalities for vectors are element-wise. Random variables are denoted in bold. The $n$-dimensional Euclidian space is denoted by $\mathbb{R}^n$. The set of non-negative integers is denoted by $\Nat$. $\E(\x)$ is used for expectation of a random variable $\x$. An $N$-dimensional vector consisting of zeros except for $i^{th}$ position is denoted by $\mathds{1}_{s_i}$.

\section{Background on Stochastic Hybrid Systems} \label{sec:shs}
In this section, we provide brief overview of a SHS construction and its mathematical characterization. The reader is referred to \cite{hespanha2006modelling,teel2014stability,Hu2000} for technical details on SHS, and its relationship with various other classes of stochastic systems.
\subsection{Basic Setup}
The state space of a SHS consists of a continuous state $\x(t)\in \mathbb{R}^{n}$ and a discrete state $\q(t) \in Q=\{s_1, s_2, \ldots, s_N\}$. There are three components of SHS that define how its states evolve over time. First, the continuous state evolves as per a stochastic differential equation (SDE)
\begin{subequations}\label{eq:SHSdef}
\begin{align} \label{eq:sde}
d\x = f(\q,\x)dt + g(\q,\x)d\w,
\end{align}
where $f: Q \times \R^n\to \R^n$ and $g:Q \times \R^n\to \R^{n\times k}$ are respectively the drift and diffusion terms, and $\w$ is a $k$--dimensional Weiner process. Second, the state $(\q,\x)$ changes stochastically through $\Resets$ transitions/resets that are characterized by the transition intensities
\begin{align} \label{eq:lambda}
&\lambda_r(\q,\x), \;\; \lambda_r: Q  \times \mathbb{R}^n  \to [0, \infty),\;\; r=1,2,\ldots,\mathcal{R}.
\end{align}
Third, the transition for each $r=1,2,\ldots,\mathcal{R}$ has an associated reset map 
\begin{align} \label{eq:phi}
& (\q,\x) \mapsto \left(\theta_r(\q), \phi_r(\q,\x)\right), \nonumber \\
& \theta_r: Q \to Q, \phi_r: Q\times \mathbb{R}^n  \to  \mathbb{R}^n
\end{align}
\end{subequations}
that defines how the pre-transition discrete and continuous states map into the post-transition discrete and continuous states. One way to think about an SHS is to consider the discrete states as different modes, each of which has an associated SDE describing the time evolution of the continuous state. The reset events can either reset the continuous state and remain in the same mode (i.e, the continuous state evolves via the same SDE as before the reset occured), or reset both the continuous state and the mode.

For purpose of this work, we first assume that for a given discrete state, the functions $f$, $g$, $\lambda_r$, and $\phi_r$ are  polynomials in $\x$. We then consider the case when these could be non-polynomial functions that are composition of rational functions, trigonometric functions, exponential, and logarithm.

%\cite{hespanha2006modelling,teel2014stability,Hu2000}.

\subsection{Extended Generator}
Mathematical characterization of SHS \eqref{eq:SHSdef} requires computation of expectation of some large class of functions evaluated on its state space.  To this end, the extended generator describes time evolution of a scalar test function $\psi: Q \times \mathbb{R}^n \to \mathbb{R}$ which is twice continuously differentiable with respect to its  second argument (i.e., $\x$). This is given as
\begin{subequations}\label{eq:ExtGen}
\begin{align}\label{eqn:ExtGenMom}
\frac{d\E \left[\psi(\q,\x)\right]}{dt}
=
\E\left[(\Leg\psi)(\q,\x)\right],
\end{align}
where $\E$ denotes the expectation operator and $\Leg$ is called the extended generator
\begin{multline}\label{eqn:shsGen}
(\Leg \psi)(\q,\x):=
\frac{\partial \psi(\q,\x)}{\partial \x} f(\q,\x) \\
 + \frac{1}{2}\, {\rm Trace} 
\left(\frac{\partial^2 \psi(\q,\x)}{\partial \x^2}g(\q,\x)g(\q,\x)^\top\right)\\
+ 
\sum_{r=1}^{\mathcal{R}}\left(\psi\left(\theta_r(\q),\phi_r(\q,\x)\right)-\psi (\q,\x) \right)\lambda_r(\q,\x).
\end{multline}
\end{subequations}
The terms $\frac{\partial \psi(\q,\x)}{\partial \x}$ and $\frac{\partial^2 \psi(\q,\x)}{\partial \x^2}$ respectively denote the gradient and the Hessian of $\psi(\q,\x)$ with respect to $\x$ \cite{hespanha2005model}. Appropriate choice of $\psi(\q,\x)$ gives a dynamics of moments of SHS as described in the next section.

\section{Moment Analysis of Polynomial SHS}\label{sec:mom}
In this section, we focus on SHS defined over polynomials: for each discrete state $\q$, the functions $f$, $g$, $\lambda_r$, and $\phi_r$ are polynomials in the continuous state $\x$. We describe how the extended generator gives time evolution of its moments. We then discuss the problem of moment closure, and propose our methodology to estimate moments.

\subsection{Moment dynamics for polynomial SHS with single discrete state}
We first consider a simpler system that has only one discrete mode/state ($\q$ can be dropped for ease of notation). For a given $n$-tuple $m=(m_1,m_2,\ldots,m_n) \in \Nat^n$, moment dynamics can be computed by plugging in the monomial test function 
\begin{equation}
\psi(\x)=
\x_1^{m_1}\x_2^{m_2}\ldots \x_n^{m_n}
\end{equation}
in \eqref{eq:ExtGen}. Here order of the moment $\E(\x_1^{m_1}\x_2^{m_2}\ldots \x_n^{m_n})$ is given by $\sum_{i=1}^{n}m_i$, and there are 
${\sum_{i=1}^{n}m_i+n-1} \choose {n-1}$ moments of the order  of order $\sum_{i=1}^{n}m_i$. 
The following standard result shows how dynamics of a collection of moments of $\x$ evolves over time for a special class of SHS that are defined via polynomials.

%Given $\x=(\x_1,\x_2,\ldots,\x_n)\in \R^n$ and an $n$-tuple of non-negative integers $m=(m_1,m_2,\ldots,m_n)$, we use multi-index notation $\x^{(m)}$ to denote the monomial $\x_1^{m_1}\x_2^{m_2}\ldots \x_n^{m_n}$. The total degree of $\x^{(m)}$ is denoted by $|m|=\sum_{i=1}^{n}m_i$. 

\begin{lemma}\label{lemma:poly1}
Let $f(\x)$, $g(\x)g(\x)^{\top}$, $\lambda_r(\x)$ and $\phi_r(\x)$ be polynomials in $\x$. Denoting the vector consisting of all moments up to a specific order of $\x$ by $\mathcal{X}$, its time evolution can be compactly written as
\begin{equation}\label{eq:singlemode}
\frac{d\XMom}{dt}=
A \XMom+B \XbarMom\\
\end{equation}
for appropriately defined matrices $A$, $B$. Here $\XbarMom$ is a collection of moments whose order is higher than those stacked up in   $\XMom$.
\end{lemma}
\begin{proof}
Since $f(\x)$, $g(\x)g(\x)^{\top}$, $\lambda_r(\x)$ and $\phi_r(\x)$ are polynomials, the extended generator in \eqref{eqn:shsGen} maps monomials of the form $\x_1^{m_1}\x_2^{m_2}\ldots \x_n^{m_n}$ to a linear combination of monomials of different orders. Collecting all moments up to some order (including the zeroth order moment) in a vector $\XMom$, 
the form in \eqref{eq:singlemode} follows from \eqref{eqn:ExtGenMom}.
%\hfill\qed
\end{proof}
The moment dynamics in \eqref{eq:singlemode} is well-known \cite{hespanha2005model}.  It is worth noting that the matrix $B$ has all its elements $zero$ if all the functions $f(\x)$, $g(\x)g(\x)^{\top}$, $\lambda_r(\x)$ and $\phi_r(\x)$ are affine in $\x$. In this case, the moments contained in $\mathcal{X}$ can be exactly computed. Next, we discuss moment dynamics for SHS with multiple discrete states.

\subsection{Moment dynamics for polynomial SHS with finite number of discrete states}
Now we consider a general SHS that has a finite, but more than one, discrete states. In this case, one is interested in knowing moments of the continuous state given a discrete state and the probability that the system is in the given discrete state. To compute these, we define an $N$-dimensional state
\begin{subequations}\label{eq:b}
\begin{equation}\label{eq:defb}
\boldb = (\boldb_1,\boldb_2,\ldots,\boldb_N)\in \R^N
\end{equation}
such that each $\boldb_i, i=1, 2, \ldots, N$ serves as an indicator of the discrete state being $\q=s_i$
\begin{equation}\label{eq:defbi}
\boldb_i=
\begin{cases}
1, & \q=s_i \\
0, & {\rm else}
\end{cases}.
\end{equation}
For example, when the discrete state $\q=s_1$, then we represent it by the tuple $\boldb=(1,0,\ldots,0)$. It follows that
the following properties hold
\begin{equation}\label{eq:Binary}
\sum_{i=1}^{N}\boldb_i=1; \quad \boldb_i \boldb_j=0, \,i\neq j; \quad \boldb_i^2=\boldb_i.
\end{equation}
\end{subequations}
Furthermore, $\E(\boldb_i)$ is equal to the probability of $\q=s_i$,
while $\E(\boldb_i \x_1^{m_1}\x_2^{m_2}\ldots \x_n^{m_n})$ is equal to the product of the the
probability that $\q=s_i$ and the moment of $\x_1^{m_1}\x_2^{m_2}\ldots \x_n^{m_n}$, conditioned on $\q=s_i$. We can recast the SHS in \eqref{eq:SHSdef} to the new state space $(\boldb,\x)$ as described via the following lemma.
\begin{lemma}\label{lemma:Nto1modes}
Consider the SHS described in \eqref{eq:SHSdef}. With $\boldb \in \R^N$ defined in \eqref{eq:b}, let a single-discrete mode SHS with state space $\left(\boldb, \x \right) \in \R^{N+n}$ be described by the continuous dynamics
\begin{subequations}\label{eq:SHSRecast}
\begin{align}\label{eq:RecastCont}
d\begin{bmatrix}\boldb \\ \x \end{bmatrix}=
\begin{bmatrix} 0 \\ \sum_{i=1}^{N}\boldb_i f(s_i,\x) \end{bmatrix} dt 
+ 
\begin{bmatrix} 0 \\ \sum_{i=1}^{N}\boldb_i g(s_i,\x)d\w \end{bmatrix},
\end{align}
reset intensities
\begin{equation}\label{eq:RecastInt}
\sum_{i=1}^{N}\boldb_i \lambda_r(s_i,\x), \quad r=1, 2, \ldots, \mathcal{R},
\end{equation}
and reset maps
\begin{equation}\label{eq:RecastRM}
\left(\boldb,\x\right)\mapsto \left(\boldb-\sum_{i=1}^{N}\boldb_i \mathds{1}_{s_i}+\sum_{i=1}^{N}\boldb_i \mathds{1}_{\theta_r(s_i)},\sum_{i=1}^{N}\boldb_i \phi_r(s_i,\x)\right).
\end{equation}
\end{subequations}
Then \eqref{eq:SHSRecast} recasts \eqref{eq:SHSdef} in $(\boldb,\x)$ space.
 \end{lemma}
\begin{proof}
Let $\q(t)=s_j \in Q$. Then \eqref{eq:b} implies that dynamics of $\x$ in \eqref{eq:RecastCont} becomes
\begin{equation}
d\x=f(s_j,\x)dt+g(s_j,\x)d\w,
\end{equation}
which is same as \eqref{eq:sde}. Likewise, the rest intensities for both  \eqref{eq:SHSRecast} and \eqref{eq:SHSdef} take the form
\begin{equation}
\lambda_r(s_j,\x), \quad r=1, 2, \ldots, \mathcal{R}.
\end{equation}
As for the reset maps, \eqref{eq:RecastRM} yields
\begin{equation}
\left(\mathds{1}_{s_j},\x\right)\mapsto \left(\mathds{1}_{s_j}-\mathds{1}_{s_j}+\mathds{1}_{\theta_r(s_j)},\phi_r(s_j,\x)\right),
\end{equation}
which by definition in \eqref{eq:b} is same as  \eqref{eq:phi}
\begin{equation}
\left(s_j,\x\right)\mapsto \left(\theta_r(s_j),\phi_r(s_j,\x)\right).
\end{equation}
Since we arbitrarily chose $\q=s_j \in Q$, the equivalence between the
two SHSs will hold true for any $\q$.
%\hfill\qed
\end{proof}

To write the  moment dynamics of SHS in \eqref{eq:SHSRecast}, we can use monomial test functions 
\begin{equation}
\psi(\boldb,\x)=
\boldb_1^{m_1} \boldb_2^{m_2}\ldots \boldb_{N}^{m_{N}}\x_1^{m_{N+1}}\x_2^{N+m_2}\ldots \x_n^{m_{N+n}},
\end{equation}
supplemented with the constraints in \eqref{eq:Binary}. It is worth noting that  \eqref{eq:SHSRecast} is a polynomial SHS in $(\boldb,\x)$ space if the original SHS was polynomial in $\x$. The following result provides a general form for the moment dynamics.
\begin{theorem}\label{lemma:polyN}
Consider the SHS in \eqref{eq:SHSRecast}. Let $f$, $g$, $\lambda_l$ and $\phi_l$ be polynomials in $\x$. Denoting the vector consisting of all moments up to a specific order of the state $(\boldb,\x)$ by $\mathcal{X}$, its time evolution can be compactly written as
\begin{subequations}\label{eq:MomMultMode}
\begin{align}
\frac{d\mathcal{X}}{dt}=
& A \mathcal{X}+B \bar{\mathcal{X}}, \label{eq:MomMultMode1}\\
0= & C \mathcal{X}+D \bar{\mathcal{X}} \label{eq:MomMultMode2}
\end{align}
\end{subequations}
for appropriately defined matrices $A$, $B$, $C$, $D$. Here $\bar{\mathcal{X}}$ is a collection of moments whose order is higher than those stacked up in   $\mathcal{X}$.
\end{theorem}
\begin{proof}
Since  \eqref{eq:SHSRecast} is polynomial in $(\boldb,\x)$, the form in \eqref{eq:MomMultMode1} follows from Lemma~\ref{lemma:poly1}. The property $\boldb_i\boldb_j=0$ in \eqref{eq:Binary} implies that for a non-zero $m_i \in \Nat$, all moments except those of the form $\E\left(\boldb_i^{m_i} \x_1^{m_{N+1}}\x_1^{m_{N+2}}\ldots \x_n^{m_{N+n}}\right)$ are zero. Furthermore, $\boldb_i^2=\boldb_i$ results in 
\begin{multline}\label{eq:highLow}
\E\left(\boldb_i^{m_i} \x_1^{m_{N+1}}\x_1^{m_{N+2}}\ldots \x_n^{m_{N+n}}\right) \\
=\E\left(\boldb_i \x_1^{m_{N+1}}\x_1^{m_{N+2}}\ldots \x_n^{m_{N+n}}\right),
\end{multline}
for all $m_i\geq 1$. The constraint $\sum_{i=1}^{N}\boldb_i=1$ results in
\begin{multline}
\sum_{i=1}^{N}\E\left(\boldb_i \x_1^{m_{N+1}}\x_1^{m_{N+2}}\ldots \x_n^{m_{N+n}}\right)\\
-\E\left(\x_1^{m_{N+1}}\x_1^{m_{N+2}}\ldots \x_n^{m_{N+n}}\right)=0.
\end{multline}
These three constraints can be compactly represented by \eqref{eq:MomMultMode2}.
\end{proof}

\begin{remark}
In Theorem~\ref{lemma:polyN} we have assumed that all moments up to a
certain order are collected in $\XMom$ and remaining, higher order,
moments are collected in $\XbarMom$. However, since many of these
moments are equal to zero, in practice we do not include them in
$\XMom$ and $\XbarMom$. Similarly, higher order moments that are equal
to lower order moments, as in \eqref{eq:highLow}, are not included. 
\end{remark}

The form of moment dynamics for polynomial SHSs implies that the moments in $\XMom$ cannot be computed exactly, since they depend upon the moments in $\XbarMom$. This is often referred to the problem of moment closure, and there are many methods that have been proposed in the literature to close the moment dynamics. Some of these methods ignore the higher order moments or cumulants to find the closure, while others use dynamical systems properties or physical principles to find the closure \cite{Kuehn2016,socha2007linearization,singh2011approximate,soltani2015conditional}. In all these methods, the approximations are ad-hoc; they could be quite accurate for a specific system under study while they could perform poorly for other systems. In the following, we discuss a semidefinite programming based method that gives provable bounds on the moments.

\subsection{Bounding Moment Dynamics}
In our recent work, we proposed to approximate the moment dynamics by making use of the fact that the higher oder cannot take arbitrary values and must conserve semidefinite properties \cite{lamperski2017analysis,lamperski2016stochastic,ghusinga2017exact}. These properties arise naturally from the fact that outer products of vectors consisting of monomials are positive semidefinite, and this semidefinite constraint is maintained by taking expectations. 
For instance, if $\x \in \R$, then
\begin{multline}
\E \left(
\begin{bmatrix}
1 \\
\x \\
\x^2
\end{bmatrix}
\begin{bmatrix}
1 &
\x &
\x^2
\end{bmatrix}
\right)\\
= 
\begin{bmatrix}
1 & \E\left(\x\right) & \E\left(\x^2\right)\\
\E\left(\x\right)& \E\left(\x^2\right) &\E\left(\x^3\right)\\
\E\left(\x^2\right) & \E\left(\x^3\right) & \E\left(\x^4\right)
\end{bmatrix}
\succeq 0.
\end{multline}

In general, if $v_1(\x),\ldots,v_p(\x)$ is an collection of
polynomials, then there is a matrix $M$ such that
\begin{equation}
  \label{eq:basicLMI}
\E\left( \begin{bmatrix}
  v_1(\x) \\
  \vdots \\
  v_p(\x)
\end{bmatrix}
\begin{bmatrix}
  v_1(\x) \\
  \vdots \\
  v_p(\x)
\end{bmatrix}^\top
\right)= M(\XMom, \XbarMom) \succeq 0,
\end{equation}
where $\mathcal{X}$ and $\bar{\mathcal{X}}$ are the collection of moments as from \eqref{eq:singlemode} \cite{lamperski2017analysis}. More constraints can be constructed by having a family of functions $h_i(\x)>0$ 
\begin{equation}\label{eq:sdpPos}
E\left(h_i(\x) \begin{bmatrix}
  v_1(\x) \\
  \vdots \\
  v_p(\x)
\end{bmatrix}
\begin{bmatrix}
  v_p(\x) \\
  \vdots \\
  v_p(\x)
\end{bmatrix}^\top
\right)= M_i (\XMom, \XbarMom) \succeq 0.
\end{equation}

Using these inequalities, bounds on moments of an SHS defined over polynomials can be computed. In particular, a lower bound on a moment of interest $\mu$ at a given time $\tau$ can be computed via the semidefinite program \cite{lamperski2017analysis}
%\begin{thm}
\begin{subequations}
    \begin{align}
      & \underset{\XMom(t),\XbarMom(t)}{\textrm{minimize}} && \mu(\tau)\\
      & \textrm{subject to} && \frac{d\XMom}{dt} = A \XMom(t) + B \XbarMom(t) \label{eqn:momdynAB}\\
       & {} && 0 = C \XMom(t) + D  \XbarMom(t) \\
      &&& M(\XMom(t), \XbarMom(t)) \succeq 0 \\
      &&& M_i(\XMom(t), \XbarMom(t)) \succeq 0 \\
      &&& \XMom(0)=\XMom_0
      \end{align}
\end{subequations}
for all $t\in [0,\tau]$.
%\end{thm}
The upper bound can be computed by maximizing the objective function. Moreover, if the number of moments stacked in $\mathcal{X}$ are increased and correspondingly the sizes of $M$ and $M_i$ are increased, the lower and upper bounds often improve. Theoretically, the increase implies that more constraints are added to the program and therefore the bounds cannot get worse. However, in practice they improve and converge to the true moment value.  

Solving the above semidefinite program however has several challenges. First, the semidefinite program needs discretization of the time in the interval $[0,\tau]$ and thereby the size of the overall program gets large quickly. Secondly, the semidefinite matrices $M$ and $M_i$ are often ill-conditioned because their elements are moments. Due to these issues, the semidefinite program based approach is computationally restrictive. Nonetheless, the program becomes much simpler if bounds on only stationary moments are desired. To see this, note that if the SHS has a stationary distribution then lower bound for a stationary moment $\mu \in \XMom$ is given by
%\begin{cor}
\begin{subequations}\label{eq:SDPss}
    \begin{align}
      & \underset{\XMom,\bar{\mathcal{X}}}{\textrm{minimize}} && \mu\\
      & \textrm{subject to} && 0= A \XMom + B \bar{\mathcal{X}} \\
       & {} && 0 = C \mathcal{X} + D \bar{\mathcal{X}} \\
      &&& M(\XMom,\XbarMom) \succeq 0\\
      &&& M_i(\XMom,\bar{\mathcal{X}}) \succeq 0
           \end{align}
           \end{subequations}
%\end{cor}
The reader may refer to \cite{meyn2012markov,deville2016moment} for details on when a stationary distribution would exist for a given stochastic process. Next, we extend the method  to study non-polynomial SHS that can be recasted as polynomial SHS with additional states and algebraic constraints.
\begin{remark}
The proposed method of estimating bounds on moments results in trivial lower bounds for systems that have all elements in the first column of $A$ as zero. For such systems, there are multiple steady-state solutions that can satisfy that bounds, and the lowest one is always the degenerate distribution. 
\end{remark}

\section{Moment Analysis for Non-Polynomial SHS}
Consider a polynomial SHS defined as in \eqref{eq:SHSdef}, with additional algebraic constraints of the form
\begin{subequations}\label{eq:const}
\begin{align}
& l_b \leq a_n(\x) \leq u_b,\\
& a_p(\x)=0.
\end{align}
\end{subequations}
where $a_n(\x)$ and $a_p(\x)$ are appropriately defined vectors. In this section, we provide a general-purpose method that can be used to cast a variety of non-polynomial SHSs to a polynomial SHSs with constraints in \eqref{eq:const}. We then extend the semidefinite programming methodology to estimate its moments.

\subsection{Moment dynamics of non-polynomial SHS  by recasting them as polynomials}
To see how various non-polynomial SHSs can be reformulated as polynomial SHSs by appending states, we first consider the SHS in \eqref{eq:SHSdef} wherein all functions are rationals except for the reset maps $\phi_l$ which we assume to be polynomial. Without loss of generality, we can consider a single discrete state since Lemma~\ref{lemma:Nto1modes} allows reduction of a SHS with multiple discrete modes. Let $J(\x)$ be the least common denominator for all $f(\x)$, $g(\x)g^\top(\x)$, and $\lambda_l(\x)$. Defining a new state $\y=\frac{1}{J(\x)}$, it is straightforward to see that one gets a polynomial SHS in the state $(\x,\y) \in \R^{n+1}$, with an equality constraint
\begin{align}
J(\x)\y-1=0.
\end{align}

While not studied formally in the context of SHSs, a similar approach to define additional states to study non-polynomial stochastic systems has been used earlier \cite{socha2007linearization,bcp16}. We propose an heuristic methodology for SHSs, which is heavily inspired from polynomial abstraction of non-polynomial deterministic hybrid systems that consist of nonlinearities involving elementary functions, viz., exponential, trigonometric, logarithm, or a composition of these  \cite{lzz15}. For simplicity we first restrict ourselves to SHSs with no resets and carry out the following steps.
\begin{enumerate}
\item[(i)] Suppose there are $L_1$ non-polynomial/non-rational functions of $\x$ in $f$, and $g$ wherein composite functions are counted as many times as they are composition of. Define new states $\y_l, l=1, 2, \ldots, L_1$, for each.
\item[(ii)] Take derivatives of each of the states $\y_l, l=1, 2, \ldots, L_1$ with respect to its arguments. If there are non-rational nonlinear terms consisting of $\x$ and $\y_l$ that are not absorbed by $\y_l$, define additional states to account for them. Suppose there are $\y_l, l=1, 2, \ldots, L_2$ states now.
\item[(iii)] Repeat step (ii) until rational terms appear. Define another state to account for the least common denominator of the rational terms. Eventually we would have $L$ additional states $\y_l, l=1, 2, \ldots, L$.
\item[(iv)] Defining some of the new variables is accompanied algebraic constraints that can be succinctly put polynomial equality constraints $a_p(\x,\y)=0$ and  polynomial inequality constraints $a_n(\x,\y)\geq 0$.
\end{enumerate}

We explain these steps via a simple example. Let state $\x \in \R$ of an SDE evolve as per
\begin{subequations}
\begin{equation}
d\x=-\exp(\sin(\log(\x)))dt+d\w.
\end{equation}
Following step (i), we defined three new states
\begin{equation}
\y_1=\log(\x), \quad \y_2=\sin(\y_1), \quad \y_3=\exp(\y_2).
\end{equation}
Next, we take derivatives of these states
\begin{equation}
\frac{d \y_1}{d\x}=\frac{1}{\x}, \quad \frac{d\y_2}{d\y_1}=\cos(\y_1), \quad \frac{d\y_3}{d\y_2}=\exp(\y_2)=\y_3.
\end{equation}
Except for $\cos(\y_1)$, other terms are in terms of rational functions of the states $\x$ and $\y_l$. As per step (iii), we define
\begin{equation}
\y_4=\cos(\y_1).
\end{equation}
Since we have $\frac{d\cos(\y_1)}{d\y_1}=\sin(\y_1)=\y_2$, we do not need to define additional states except for one to absorb the least common denominator of the rational terms
\begin{equation}
\y_5=\frac{1}{\x}.
\end{equation}
By definition of these states, we can obtain algebraic constraints as
\begin{equation}
\y_2^2+\y_4^2-1=0, \quad \x \y_5-1=0,
\end{equation}
where the first constraint arises from the trigonometric identity relating $\sin$ and $\cos$ while the second constraint arises from the definition of $\y_5$. These states also allows one to get some constraints such as 
\begin{equation}
0<\x, \quad -1 \leq \y_2 \leq 1, \quad 0 < \y_3, \quad -1 \leq \y_4 \leq 1.
\end{equation}
\end{subequations}

Recall the extended generator in \eqref{eqn:shsGen}. For rational SHSs with polynomial reset maps, or SHSs with elementary functions but no resets, the above recipe ensures that a monomial of the form 
$$\x_1^{m_1}\x_2^{m_2}\ldots \x_n^{m_n} \y_1^{m_{n+1}} \y_2^{m_{n+2}} \ldots \y_L^{m_{n+L}}$$
is mapped to monomials of the same form. This is because the state space is closed under derivatives.  
\begin{remark}
 For SHSs with rational functions, if the reset map is rational then each monomial in $\x$ is mapped to a different rational function and a lot many additional states may be required to define moment dynamics up to a certain order. In light of this, the above method may seem bit restrictive, but in practice there are numerous examples of SHSs wherein only polynomial reset maps appear.  Likewise, for elementary nonlinearities, we considered only SHSs that have no resets. However, the setup may be extended to include reset maps as long as the state space is closed under derivatives. For example, if $f$, $g$, $\lambda_l$  consist of $\exp(\x)$ for $\x \in \R$, then simple reset maps such as $\x \mapsto c_1 \x + c_2$ fall under this category.
\end{remark}
In the following Lemma, we provide a general form of moment dynamics for non-polynomial SHSs that can be casted as a polynomial SHS with constraints of the form \eqref{eq:const}. 
\begin{theorem}
Consider a single discrete mode SHS in \eqref{eq:SHSdef} with constraints of the form \eqref{eq:const}. Collecting moments of the state space $(\x,\y)$ up to a specific order in the vector $\XMom$, the moment dynamics is given by
\begin{subequations}
\begin{align}
\frac{d\mathcal{X}}{dt}=
& A \mathcal{X}+B \bar{\mathcal{X}},  \label{eq:MomMultModeGeneral1}\\\
0= & C_p \mathcal{X}+D_p \bar{\mathcal{X}}  \label{eq:MomMultModeGeneral2}\
\end{align}
\end{subequations}
where $\XbarMom$ contains moments of higher order and the matrices $A$, $B$, $C_p$, $D_p$ are appropriately defined.
\end{theorem}
\begin{proof}
Since the test function is monomial of the form
$(\x_1^{m_1}\x_2^{m_2}\ldots \x_n^{m_n} \y_1^{m_{n+1}} \y_2^{m_{n+2}}
\ldots \y_L^{m_{n+L}})$, these are closed under the extended
generator. Thus, \eqref{eq:MomMultModeGeneral1} follows from
Lemma~\ref{lemma:poly1}. The algebraic constraints of the form
$a_p(\x,\y)=0$ imply that the moments in which elements of
$a_p(\x,\y)=0$  appear are equal to zero. Similar to
\eqref{eq:MomMultMode2}, these are encoded as
\eqref{eq:MomMultModeGeneral2}.
%\hfill\qed
\end{proof}
We can straightforwardly extend the above form of moment dynamics to an SHS with multiple discrete states by virtue of Lemma~\ref{lemma:Nto1modes}.

\subsection{Bounds on moments via semidefinite programming}
The preceding discussion provides a recipe to write a non-polynomial SHS as polynomial SHS with algebraic constraints consisting of both equalities and inequalities. While we have incorporated the equality constraints in moment dynamics in \eqref{eq:MomMultModeGeneral2}, the inequality constraints remain to be incorporated. Recall the constraints of obtained from \eqref{eq:sdpPos} have positive polynomials $h_i(\x)$ that can absorb inequalities. We can thus embed the constraints $a_n(\x,\y)$ in the matrices $M_i(\XMom,\XbarMom)$. Formally the semidefinite program is  given by
\begin{subequations}
    \begin{align}
      & \underset{\mathcal{X},\bar{\mathcal{X}}}{\textrm{minimize}} && \mu\\
      & \textrm{subject to} && 0= A \mathcal{X} + B \bar{\mathcal{X}} \\
       & {} && 0 = C_p \mathcal{X} + D_p \bar{\mathcal{X}} \\
      &&& M(\mathcal{X},\bar{\mathcal{X}}) \succeq 0\\
      &&& M_i(\mathcal{X},\bar{\mathcal{X}}) \succeq 0
           \end{align}
           \end{subequations}
As mentioned earlier, if a multimode SHS were to be considered, the form of SDP remains to be similar with another constraint $C\XMom+D \XbarMom=0$ being added.
\section{Numerical Examples}\label{sec:numerical} 
We illustrate our approach using two examples. The first example comprises of multiple discrete states and polynomial dynamics/resets, and the second example consists of a single discrete state with rational dynamics.

\begin{example}[TCP On-Off \cite{hespanha2006modelling,hespanha2005model}.]
\label{ex:tcpOnOff}
\begin{figure}[h]
\centering
\includegraphics[width=0.8\linewidth]{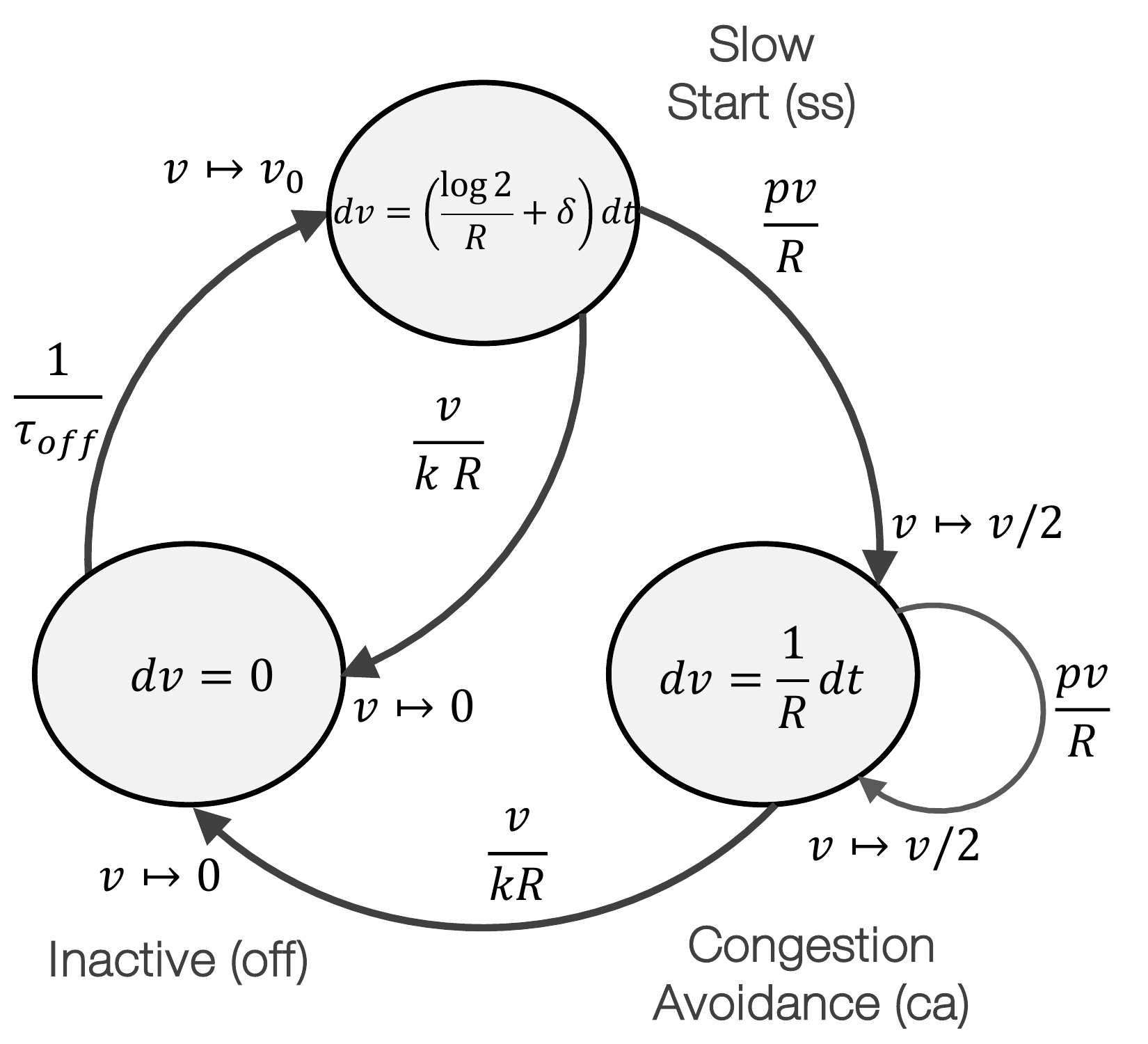}
\caption{Stochastic Hybrid System representation of TCP On-Off model. Here, there are three discrete modes and the continuous dynamics $\boldv$ evolves as per different differential equations depending upon which mode the system is operating in. Various reset intensities and reset maps are also shown.}
\label{fig:tcpOnOffSHS}
\end{figure}

We consider a simple version of the TCP on-off model. Here, the continuous state of the model is denoted by $v$, which represents the congestion window size of the TCP. The model consists of three discrete states, namely, $\{off, ss, ca\}$, which stand for off, slow start, and congestion avoidance, respectively. 

During these modes, the continuous-state evolves as
\begin{equation}
\label{eq:onoffDyn}
f(\q,\boldv,t)=
\begin{cases}
0, & \q=off \\
\frac{\log{2}}{R}\boldv + \delta, & \q=ss \\
\frac{1}{R}, & \q=ca\\
\end{cases}
\end{equation}

The transitions between the discrete modes are of three types: drop occurences, which correspond to transitions from the ss and ca modes to the ca mode; start of new flow, which correspond to the transitions from the off mode to the ss mode; and termination of flows, which correspond to transitions from the ss and ca modes to the off mode. These transitions are described via the reset maps 
\begin{align}
\label{eq:TCPmap1}
& \phi_{drop}(\q,\boldv)=
\begin{cases}
\left(ca,\frac{\boldv}{2}\right), & \q \in \{ss, ca\} \\
(off,\boldv), & \q = off
\end{cases} \\
\label{eq:TCPmap2}
& \phi_{start}(\q,\boldv)=
\begin{cases}
\left(\q,\boldv \right), & \q \in \{ss, ca\} \\
(ss,v_0), & \q = off
\end{cases} \\
\label{eq:TCPmap3}
& \phi_{end}(\q,\boldv)=
\begin{cases}
\left(off,0\right), & \q \in \{ss, ca\} \\
(off,\boldv), & \q = off
\end{cases}
\end{align}
with reset intensities
\begin{align}
\label{eq:TCPrate1}
& \lambda_{drop}(\q,\boldv)=
\begin{cases}
\frac{p \boldv}{R}, & \q \in \{ss, ca\} \\
0, & \q = off
\end{cases}\\
\label{eq:TCPrate2}
& \lambda_{start}(\q,\boldv)=
\begin{cases}
0, & \q \in \{ss, ca\} \\
\frac{1}{\tau_{off}}, & \q = off
\end{cases} \\
\label{eq:TCPrate3}
& \lambda_{end}(\q,\boldv)=
\begin{cases}
\frac{\boldv}{k R}, & \q \in \{ss, ca\} \\
0, & \q = off
\end{cases}.
\end{align}
Here $R$ is the round trip time, $p$ is the packet drop rate parameter. \\
\end{example}
To write moment dynamics, we define the indicator state variables $\boldb_{ss}$, $\boldb_{ca}$, and $\boldb_{off}$ as in \eqref{eq:defbi}. The resulting single-mode SHS is shown in Fig.~\ref{fig:tcpOneMode}.
\begin{figure}[h]
\centering
\includegraphics[width=0.9\linewidth]{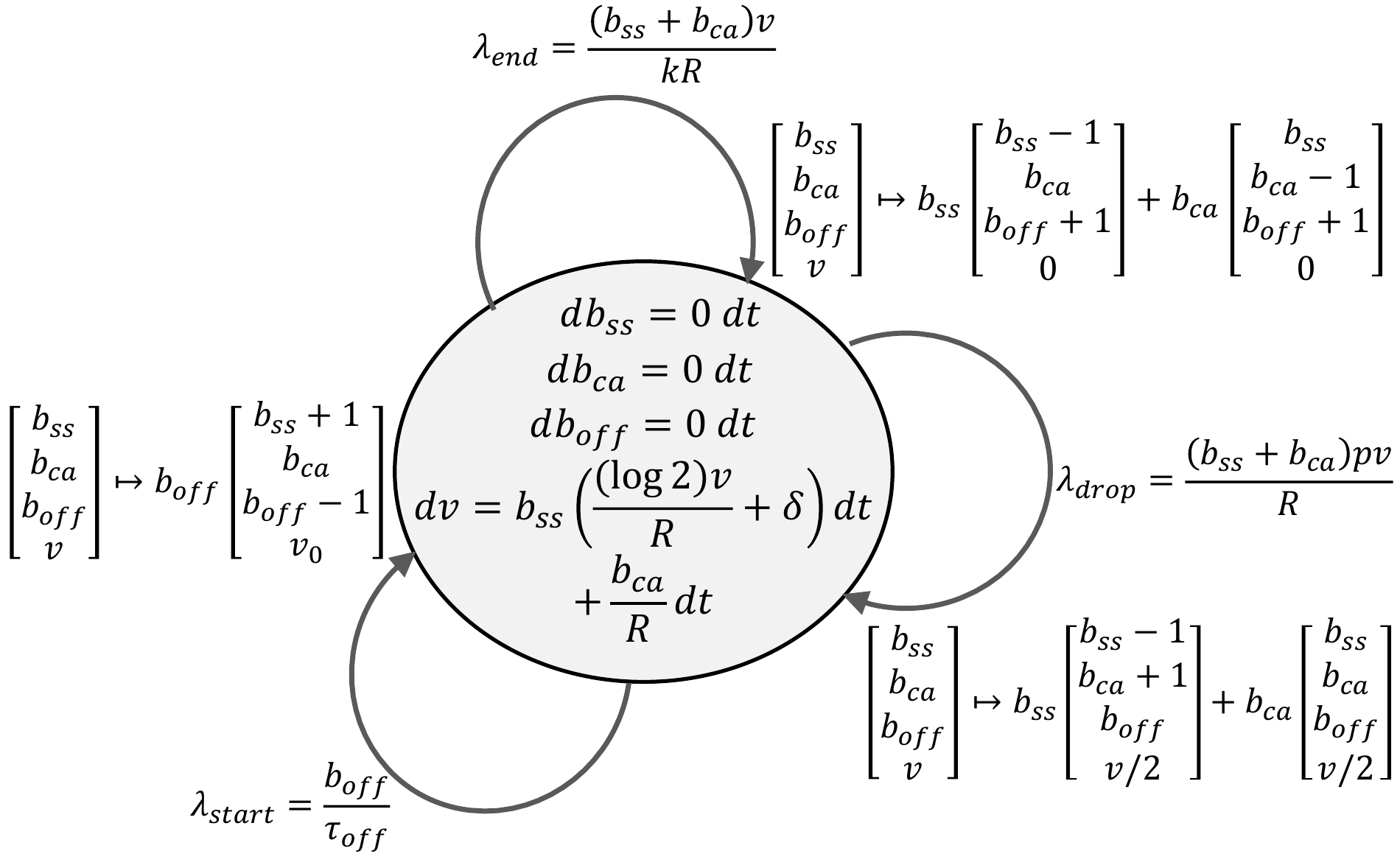}
\caption{An equivalent single-mode representation of the Stochastic Hybrid System representation of TCP On-Off model.}
\label{fig:tcpOneMode}
\end{figure}
Using extended generator, we write dynamics of the non-zero moments. In particular, we have
\begin{subequations}
\begin{multline}
\frac{d\E(\boldb_{ss}\boldv^m)}{dt}=\frac{m \log(2)}{R} \E\left(\boldb_{ss} \boldv^m\right)+m \delta \E\left(\boldb_{ss} \boldv^{m-1}\right) \\
+\frac{v_0^m}{\tau_{off}}\E(\boldb_{off})-\left(p+\frac{1}{k}\right)\frac{\E(\boldb_{ss}\boldv^{m+1})}{kR},
\end{multline}
\begin{multline}
\frac{d\E(\boldb_{ca}\boldv^m)}{dt}=\frac{m}{R} \E\left(\boldb_{ca} \boldv^{m-1}\right)+\frac{p}{2^{m}R}\E(\boldb_{ss}\boldv^{m+1}) \\
-\left(\frac{p(2^m-1)}{2^m R}+\frac{1}{kR}\right)\E(\boldb_{ca}\boldv^{m+1}),
\end{multline}
\begin{multline}
\frac{d\E(\boldb_{off}\boldv^m)}{dt}=-\frac{\E(\boldb_{off}\boldv^m)}{\tau_{off}}+\frac{\E(\boldb_{ss}\boldv^{m+1})}{kR} \\
+\frac{\E(\boldb_{ca}\boldv^{m+1})}{kR} ,
\end{multline}
for $m \in \Nat$. Using these moment equations along with the semidefinite constraints and algebraic constraints arising from the definition of $\boldb_{ss}, \boldb_{ca},\boldb_{off}$, the semidefinite program as in \eqref{eq:SDPss} can be set up. We can also generate matrices $M_i$ by using non-negativity of $\boldb_{ss}, \boldb_{ca},\boldb_{off}$, $1-\boldb_{ss},1-\boldb_{ca},1-\boldb_{off}$.
\end{subequations}
Taking specific values of $R=5$, $\tau_{off}=0.5$, $k=20$, $p=0.05$, $v_0=1$, we get $0.0252 \leq \E \left(\boldb_{ss}\right) \leq 1$ by utilizing moments of order $2$. Considering higher order moments improves these estimates, and we get $0.0912 \leq \E \left(\boldb_{ss}\right) \leq 0.115$ for moments of order $7$ (see Fig.~\ref{fig:tcpOnOffResults}). 
\begin{figure}[h]
\centering
\includegraphics[width=\linewidth]{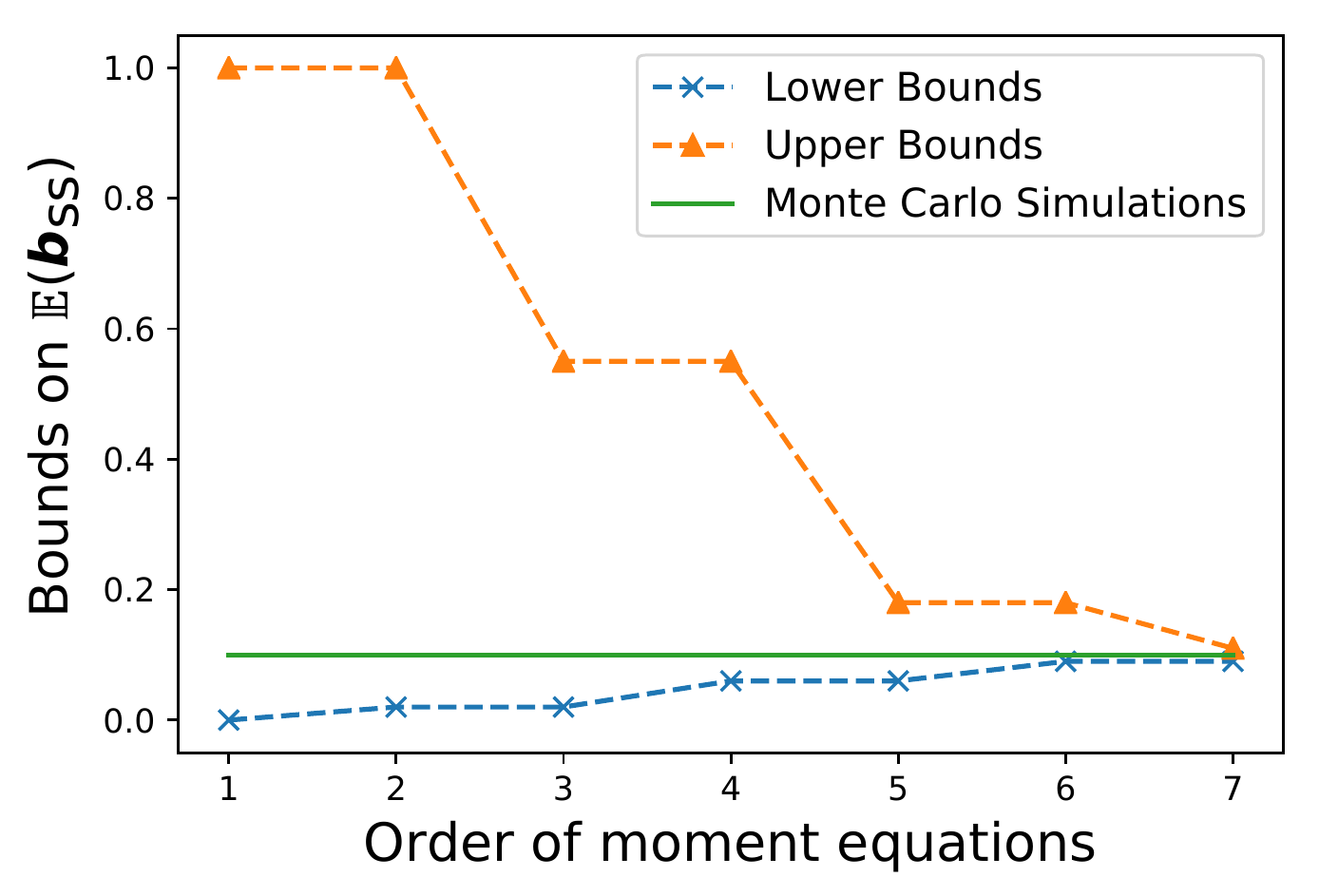}
\caption{Bounds on $\E(\boldb_{ss})$ (i.e., the probability that the system is in the mode $ss$) for the TCP on-off example. The bounds improve and converge to the true value of the moment as the order of moments used in the semidefinite program is increased.}
\label{fig:tcpOnOffResults}
\end{figure}

\begin{example}[Cell division]
\label{ex:celldiv}
An ubiquitous feature of living cells is their growth and subsequent division in daughter cells. Several models have been proposed to explain how growing cells decide to divide \cite{wrp10,tes12a,mys12,rhk14,modi2017analysis}. Here, we consider a model wherein the cell size grows as per the differential equation
\begin{equation}\label{eq:celldivDyn}
d\boldv=\left(\alpha_1+\frac{\alpha_1 \boldv}{\boldv+v_1}\right)dt.
\end{equation}
This setup encompasses both the linear growth of cell size (if $\alpha_2=0$ or if $v_1=0$) and the exponential growth (if $\boldv \ll v_1$). We assume that the cell divides as per a size-dependent rate
\begin{equation}\label{eq:celldivRate}
\lambda(\boldv) = (\boldv/v_2)^n
\end{equation}
This rate is analogous to the so-called sizer strategy in the limit when $n \to \infty$ wherein the cell divides as it attains a critical volume $v_2$. A finite value of $n$ represents imperfect implementation of a sizer model. Upon the reset, the cell size is reset to 
\begin{equation}\label{eq:celldivReset}
\phi(v) = \frac{\boldv}{2}.
\end{equation}
\begin{figure}[h]
\centering
\includegraphics[width=0.6\linewidth]{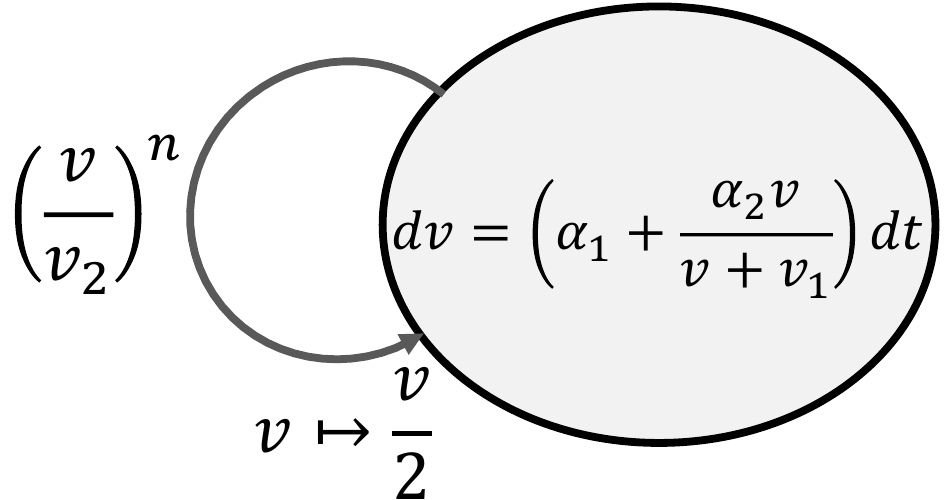}
\caption{Stochastic Hybrid System representation of sizer model of cell division. The cell size $v$ grows as per the deterministic differential equation that is a combination of two growth regimes, a linear growth with parameter $\alpha_1$ and a saturating exponential growth with parameter $\alpha_2$. The cell divides with intensity $(\boldv/v_2)^n$ and the size resets to $v/2$ (i.e., size divides in two daughters). The parameter $n$ represents imperfect implementation of the sizer and the cell divides at attainment of volume $v_2$ as $n \to \infty$.}
\label{fig:celldivSHS}
\end{figure}
\end{example}
Since the dynamics contains a rational function, we define a new state $\y=\frac{1}{\boldv+v_1}$. The SHS can then be recasted as polynomial SHS with the new continuous dynamics
\begin{equation}
d\boldv=\left(\alpha_1+\alpha_1 \boldv\y \right)dt,
\end{equation}
and an algebraic constraint
\begin{equation}
\boldv \y + v_2 \y -1 =0.
\end{equation}
The dynamics of the moment of a form $\E(\boldv^{m_1})$ can be computed as
\begin{multline}\label{eq:celldivGen}
\frac{d\E(\boldv^{m_1})}{dt}=
m_1 \alpha_1 \E(\boldv^{m_1-1})+\\
m_1 \alpha_2 \E(\boldv^{m_1}\y)
-\frac{2^m-1}{2^m v_2^n} \E \left( \boldv^{m_1+n}\y^{m_2}\right).
\end{multline}
These moment equations can be used along with the semidefinite constraints obtained from joint moments of the form $\E(\boldv^{m_1}\y^{m_2})$ and utilizing the algebraic constraints $\boldv \y + v_2 \y -1 =0$. 

As in the previous example, here too we can solve the steady-state moment equations. The technique can be used to explore the effect of parameters in noise in cell size. To this end, we plot the noise in cell size as a function of the cell size exponent $n$ in Fig.~\ref{fig:celldivResults}. Our results show that the cell size noise decreases with increase in $n$, which is expected since the size control on when the division should take place becomes stronger. Similar results were obtained in \cite{vargas2016conditions}, albeit for only exponential growth rate strategy and polynomial dynamics.
\begin{figure}[h]
\centering
\includegraphics[width=\linewidth]{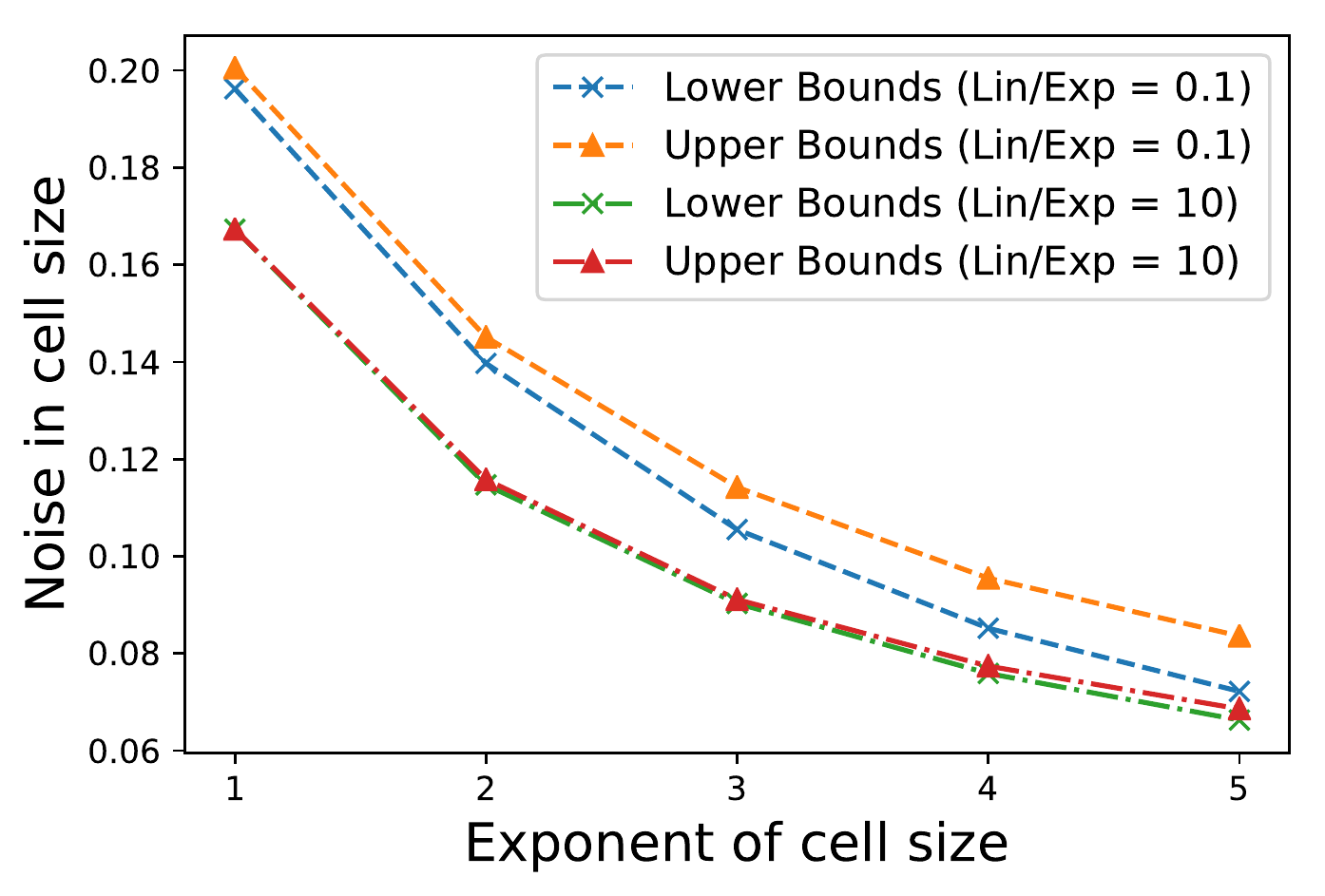}
\caption{Bounds on noise (quantified via coefficient of variation squared) in cell size as a function of cell size exponent. Using ten moment equations, the bounds are computed via semidefinite program for different exponents of cell size. It is seen that the noise in cell size decreases with increase in the exponent. Moreover, the noise is lower when the linear growth coefficient $\alpha_1$ is greater than the exponential growth coefficient $\alpha_2$. }
\label{fig:celldivResults}
\end{figure}

\section{Conclusion}\label{sec:conc} 
Stochastic hybrid systems (SHSs) consist of both discrete and continuous states. Their formal probabilistic analysis via the forward Kolmogorov equation is often analytically intractable. As an alternate, often the dynamics of its statistical moments is used to compute a few lower order moment to study the system. However, the moments themselves are generally described via an infinite dimensional coupled differential equations which cannot be solved for a few lower order moments without knowing the higher order moments. This problem is known as the moment closure problem and has been a subject of extensive study in the applied mathematics literature. In this paper, we presented a semidefinite programming based method to compute exact bounds on the moments of an SHS, and illustrated its utility in computing stationary moments of SHS defined by both polynomials and non-polynomials.
Although theoretically our method  computes bounds on both transient and stationary moments, its applicability is limited since the semidefinite programs do not scale very well. Our focus of future research would be to improve scalibility of the technique.

\section*{Acknowledgment}
AS is supported by the National Science Foundation Grant ECCS-1711548.

\bibliography{References,sample}

\end{document}